\documentclass{article}

\usepackage{amsmath,amssymb,amsthm,amsrefs}
\usepackage{enumerate}
\usepackage{pifont}
\usepackage{setspace} 
\usepackage[utf8]{inputenc}

\newcommand{\Z}{\mathbb{Z}}
\newcommand{\D}{\mathbb{D}}
 
\newcommand{\N}{\mathbb{N}} 
\newcommand{\R}{\mathbb{R}} 
\newcommand{\T}{\mathbb{T}} 
\newcommand{\C}{\mathbb{C}}

\newcommand{\BC}{\mathcal{B}}
\newcommand{\CC}{\mathcal{C}}

\newcommand{\HC}{\mathcal{H}}

\newcommand{\LC}{\mathcal{L}}

\newcommand{\NC}{\mathcal{N}}
\newcommand{\OC}{\mathcal{O}}
\newcommand{\PC}{\mathcal{P}}

\newcommand{\SC}{\mathcal{S}}

\newcommand{\WC}{\mathcal{W}}
\newcommand{\XC}{\mathcal{X}}
\newcommand{\YC}{\mathcal{Y}}

\newcommand{\Hol}{\textnormal{Hol}}
\newcommand{\tr}{\textnormal{tr}}
\newcommand{\diag}{\textnormal{diag}}
\newcommand{\overbar}[1]{\mkern 1.0mu\overline{\mkern-1.0mu#1\mkern-1.0mu}\mkern 1.0mu}
\newcommand*\conj[1]{\overbar{#1}}
\newcommand*\conjvar[1]{\tilde{#1}}

\newtheorem{theorem}{Theorem}[section]
\newtheorem{lemma}[theorem]{Lemma}
\newtheorem{definition}[theorem]{Definition}
\newtheorem{proposition}[theorem]{Proposition}
\newtheorem{corollary}[theorem]{Corollary}

\newtheorem{remark}[theorem]{Remark}

\title{Vectorial Hankel operators, Carleson embeddings, and notions of $\textrm{BMOA}$}
\author{Eskil Rydhe}
\date{\today}

\usepackage[textwidth=136mm]{geometry}

\begin{document}
	\maketitle
	
	\begin{abstract}
		Let $\textrm{BMOA}_{\NC\PC}\left(\LC\right)$ denote the space of $\LC$-valued analytic functions $\phi$ for which the Hankel operator $\Gamma_\phi$ is $H^2\left(\HC\right)$-bounded. Obtaining concrete characterizations of $\textrm{BMOA}_{\NC\PC}\left(\LC\right)$ has proven to be notoriously hard. Let $D^\alpha$ denote fractional differentiation. Motivated originally by control theory, we characterize $H^2\left(\HC\right)$-boundedness of $D^\alpha \Gamma_\phi$, where $\alpha >0$, in terms of a natural anti-analytic Carleson embedding condition. We obtain three notable corollaries: The first is that $\textrm{BMOA}_{\NC\PC}\left(\LC\right)$ is \emph{not} characterized by said embedding condition. The second is that when we add an adjoint embedding condition, we obtain a sufficient but not necessary condition for boundedness of $\Gamma_\phi$. The third is that there exists a bounded analytic function for which the associated anti-analytic Carleson embedding is unbounded. As a consequence, boundedness of an analytic Carleson embedding does not imply that the anti-analytic ditto is bounded. This answers a question by Nazarov, Pisier, Treil, and Volberg.
	\end{abstract}
	
	\section{Introduction}
	Throughout this paper we let $\HC$ denote a separable Hilbert space with inner product $\left\langle\cdot,\cdot\right\rangle_\HC$. Unless we explicitly state otherwise, we assume that $\HC$ is infinite-dimensional. We denote by $\LC=\LC\left(\HC\right)$ the space of bounded linear transformations on $\HC$, by $\SC^1$ the corresponding trace class, and by $\SC^2$ the Hilbert--Schmidt class. $\XC$ will be used as a generic notation for an element of the set $\left\{\HC,\LC,\SC^1,\SC^2\right\}$. 
	
	We will use $\YC$ to denote a general Banach space. By $\Hol\left(\YC\right)$ we denote the space of $\YC$-valued analytic functions on the open unit disc $\D$. For $f\in\Hol(\YC)$, we denote the $n$th Taylor coefficients at the origin by $\hat f(n)$. We denote by $\OC\left(\YC\right)$ the space of functions in $\Hol\left(\YC\right)$ that admit an analytic extension to a larger disc (centered at the origin).  If $\YC=\C$, then we suppress this in our notation, i.e. $\Hol=\Hol\left(\C\right)$, and $\OC=\OC\left(\C\right)$. The same principle will apply to all function spaces discussed below. 
	
	For $p\in\left[1,\infty\right]$ and $\XC\in\left\{\HC,\SC^1\right\}$, we let $L^p\left(\T,\XC\right)$ denote the standard space of $p$-Bochner--Lebesgue integrable functions from $\T$ to $\XC$. Here $\T$ denotes the unit circle in $\C$. Similarly, we define $L^p\left(\T,\LC\right)$ as the natural WOT-analogue of $L^p(\T)$: A function $f:\T\to\LC$ belongs to $L^p\left(\T,\LC\right)$ if and only if for all $x,y\in\HC$ the function $\left\langle f\left(\cdot\right) x,y\right\rangle_\HC$ is measurable and, moreover, $\left\|f\right\|_{L^p\left(\T,\LC\right)}^p = \int_{\T}\left\|f\right\|_\LC^p\, dm<\infty$. Here $m$ denotes normalized Lebesgue measure on $\T$.
	
	The Hardy space $H^p\left(\XC\right)$ is the space of $f\in\Hol\left(\XC\right)$ such that
	\begin{equation}\label{Eq:Hp-norm}
	\left\|f\right\|_{H^p\left(\XC\right)}^p=\sup_{0<r<1}\left\|f_r\right\|_{L^p\left(\T,\XC\right)}<\infty,
	\end{equation}
	where we have defined the function $f_r:z\mapsto f\left(rz\right)$. An important property of Hardy space functions is that they have boundary values in a natural sense, cf. Proposition \ref{Proposition:BoundaryIdentification}. We denote the boundary values of $f\in H^p\left(\XC\right)$ by $bf\in L^p(\T,\XC)$. 
	
	The space $H^2\left(\HC\right)$ is a Hilbert space, with inner product $\left\langle f,g\right\rangle = \sum_{0}^\infty \langle \hat f(n),\hat g(n)\rangle_\HC$. Of particular importance will be the set of $H^2\left(\HC\right)$-normalized functions in $\OC\left(\HC\right)$, which we denote by $\OC_1\left(\HC\right)$. 
		
	We now introduce the main topics of this paper. Initially, we consider the scalar setting, rather than the proper vectorial one.
	
	\subsection{Hankel operators}
	Given $\phi\in\Hol$ and $f\in\OC$, we define the action of the Hankel operator $\Gamma_\phi$ on $f$ by
	\begin{equation}\label{Eq:HardyHankelFormula}
	\Gamma_\phi f\left(z\right)=\sum_{n=0}^\infty  \left(\sum_{m=0}^\infty\hat \phi\left(m+n\right)\hat f\left(m\right)  \right)z^n,\quad z\in\D. 
	\end{equation}
	A standard reference on Hankel operators is \cite{Peller2003:HankOpsBook}. We refer to $\phi$ as the symbol of $\Gamma_\phi$. We say that $\Gamma_\phi$ is bounded if it extends to a bounded operator on $H^2$. 
	
	For $\Gamma_\phi$ to be bounded it is necessary for $\phi$ to be in $H^2$. For $\phi\in H^2$, one shows by computation that $\Gamma_\phi f=P_+  \left(\phi  \conjvar{f}  \right)$, where $P_+$ denotes the orthogonal projection from $L^2\left(\T\right)$ onto $H^2$, and $\conjvar{ f}:z\mapsto f\left(\conj{ z}\right)$. 
	
	It is convenient to define the operation of coefficient conjugation, $f\mapsto f^\#$, $f^\# (z)=\conj{f(\conj z)}$. Note that this is an isomorphism on $H^2$. A classical result is that $H^1=H^2\cdot H^2$: If $f,g\in H^2$, then $f\cdot g\in H^1$, and $\left\|h\right\|_{H^1}\le \left\|f\right\|_{H^2}\left\|g\right\|_{H^2}$. Conversely, if $h\in H^1$, then there exists $f,g\in H^2$ such that $h=f\cdot g$ and $\left\|f\right\|_{H^2}\left\|g\right\|_{H^2}\le C\left\|h\right\|_{H^1}$, where $C>0$ is a constant independent of $f$ and $g$. Now choose $f$ so that $f^\# g = h$. By the calculation
	\[
	\left\langle\Gamma_\phi f,g\right\rangle = \left\langle P_+ \left(\phi\conjvar{f} \right),g\right\rangle= \left\langle  \phi\conjvar{f},g\right\rangle= \left\langle\phi,f^\#g\right\rangle = \left\langle\phi,h\right\rangle,
	\]
	one obtains that $\Gamma_\phi$ is bounded if and only if $\phi\in\left(H^1\right)^*$. 
	
	Since $H^1$ may be identified with a subspace of $L^1(\T)$, and $\left(L^1\left(\T\right)\right)^*=L^\infty\left(\T\right)$, a straightforward application of the Hahn--Banach theorem shows that $\left(H^1\right)^*=P_+L^\infty\left(\T\right)$. The fact that $\Gamma_\phi$ is bounded if and only if $\phi\in P_+L^\infty\left(\T\right)$ is known as Nehari's theorem \cite{Nehari1957:BddBilinFrms}.
	
	\subsection{Carleson embeddings}
	Every Borel measure $\mu\ge 0$ on $\D$ corresponds to a so-called Carleson embedding $H^2\hookrightarrow L^2(\D,d\mu)$. It is a classical result \citelist{\cite{Carleson1958:InterpolProblBddAnalFcns}\cite{Carleson1962:InterpolBddAnalFcnsCoronaProbl}} in complex and harmonic analysis that boundedness of such embeddings can be characterized by a simple geometric property of $\mu$. Specifically, the Carleson embedding condition
	\begin{equation}
	\sup_{f\in\OC_1\left(\HC\right)} \int_\D   \left|f\left(z\right) \right|^2  \, d\mu\left(z\right)<\infty
	\end{equation}
	holds if and only if $\mu$ satisfies the so-called Carleson intensity condition
	\begin{equation}\label{Eq:CarlesonInt}
	\sup_{\substack{I\subset\T\\ I\textnormal{ arc}}}\frac{\mu\left(\left\{w\in\D;\, 1-m(I)<\left|w\right|<1,\, \frac{w}{\left|w\right|}\in I\right\}\right)}{m\left(I\right)}<\infty.
	\end{equation}
	
	\subsection{Bounded mean oscillation}
	A bridge connecting Hankel operators, and Carleson embeddings is given by $\textrm{BMOA}$; bounded mean oscillation of analytic functions. Suppose that $\phi\in H^1$. We then say $\phi$ belongs to the class $\textrm{BMOA}$ if and only if
	\[
	\left\|\phi\right\|_*=\sup_{\substack{I\subset\T\\ I\textnormal{ arc}}}\frac{1}{m\left(I\right)}\int_I   \left|b\phi-\left(b\phi\right)_I \right|\, dm<\infty.
	\]
	Here $\left(b\phi\right)_I$ denotes the Lebesgue integral average $\frac{1}{m\left(I\right)}\int_I b\phi\, dm$. The quantity $\left\|\cdot\right\|_*$ is a semi-norm. The class $\textrm{BMOA}$ becomes a Banach space when equipped with the norm $\left\|\phi\right\|_{\textrm{BMOA}}=\left|\phi\left(0\right)\right|+\left\|\phi\right\|_*$.
	
	A celebrated result by Fefferman \citelist{\cite{Fefferman1971:CharBMO}\cite{Fefferman-Stein1972:HpSpaces}} is that $\textrm{BMOA}$ is in fact the dual of $H^1$. Moreover, $\phi\in \textrm{BMOA}$ if and only if the measure $\mu$ given by $d\mu = \left|\phi'(z)\right|^2\left(1-\left|z\right|^2\right)\, dA(z)$ satisfies \eqref{Eq:CarlesonInt}. As a summary of this discussion we have:
	\begin{proposition}
		Let $\phi\in H^1$. Then the following are equivalent:
		\begin{itemize}
			\item[$(i)$] $\Gamma_\phi$ is $H^2$-bounded.
			\item[$(ii)$] $\phi\in\left(H^1\right)^*$.
			\item[$(iii)$] $\phi\in P_+L^\infty\left(\T\right)$.
			\item[$(iv)$] $\phi\in \textrm{BMOA}$.
			\item[$(v)$] The measure given by $d\mu = \left|\phi'(z)\right|^2\left(1-\left|z\right|^2\right)\, dA(z)$ has finite Carleson intensity.
			\item[$(vi)$] The Carleson embedding $H^2\hookrightarrow L^2\left(\D,\left|\phi'(z)\right|^2\left(1-\left|z\right|^2\right)\, dA\left(z\right)\right)$ is bounded.
		\end{itemize}
	\end{proposition}
	
	\subsection{The vectorial setting}
	Note that \eqref{Eq:HardyHankelFormula} makes perfect sense if $\phi\in\Hol(\LC)$ and $f\in\OC(\HC)$. We take this as the definition of a vectorial Hankel operator $\Gamma_\phi$. The factorization result $H^1\left(\SC^1\right)=H^2\left(\SC^2\right)\cdot H^2\left(\SC^2\right)$, due to Sarason \cite{Sarason1967:GenInterpol}, implies that $\Gamma_\phi$ is $H^2(\HC)$-bounded if and only if $\phi\in \left(H^1\left(\SC^1\right)\right)^*$, very much like in the scalar setting. 
	
	Since $\left(L^1\left(\T,\SC^1\right)\right)^*$ is not equal to $L^\infty(\T,\LC)$ ($\LC$ does not have the so-called Radon--Nikodym property, e.g. \cite{Diestel-Uhl1977:VecMeasures}), it is not obvious that $\left(L^1\left(\T,\SC^1\right)\right)^*=P_+L^\infty\left(\T,\LC\right)$. However, this follows from a vectorial extension of Nehari's theorem, due to Page \cite{Page1970:BddCompctVecHankOps}: $\Gamma_\phi$ is $H^2\left(\HC\right)$-bounded if and only if $\phi\in P_+L^\infty\left(\T,\LC\right)$.
		
	The space of $\LC$-valued analytic functions for which the corresponding Hankel operators are $H^2\left(\HC\right)$-bounded is commonly referred to as \emph{Nehari--Page} $\textrm{BMOA}$: 
	\begin{definition}\label{def:NP}
		Let $\phi\in\Hol\left(\LC\right)$. We then say that $\phi\in \textrm{BMOA}_{\NC\PC}\left(\LC\right)$ if and only if
		\[
		\left\|\phi\right\|_{\textrm{BMOA}_{\NC\PC}}=\sup_{f\in\OC_1\left(\HC\right)}\left\|\Gamma_\phi f\right\|_{H^2\left(\HC\right)}<\infty.
		\]
	\end{definition}
		
	While $\textrm{BMOA}_{\NC\PC}\left(\LC\right)$ can be identified either with $P_+L^\infty\left(\T,\LC\right)$ or with $\left(H^1\left(\SC^1\right)\right)^*$, these characterizations are of an abstract nature. Finding concrete conditions that characterize $\textrm{BMOA}_{\NC\PC}\left(\LC\right)$ has proven to be notoriously difficult. For example, define the class $\textrm{BMOA}_\OC\left(\LC\right)$ as the class of $\phi\in H^1\left(\LC\right)$ such that the oscillation condition
	\[
	\left\|\phi\right\|_*=\sup_{\substack{I\subset\T\\ I\textnormal{ arc}}}\frac{1}{m\left(I\right)}\int_I\left\|b\phi-\left(b\phi\right)_I\right\|_\XC\, dm<\infty
	\]
	holds. Then
	\[
	\textrm{BMOA}_\OC\left(\LC\right)\subsetneq \textrm{BMOA}_{\NC\PC}\left(\LC\right).
	\]
	This fact represents an area of research, where authors consider some aspect of the theory for scalar-valued $\textrm{BMOA}$ (or its harmonic or dyadic analogues), and then discuss to what extent this aspect carries over to the vector-valued case, e.g. \citelist{\cite{Blasco1988:HardySpacesVecValDuality}\cite{Blasco-Pott2008:EmbOpValDyadicBMO}\cite{Bourgain1986:VecValSingIntsHardy-BMODualityChapter}\cite{Gillespie-Pott-Treil-Volberg2004:LogGrowthHilbTransfVecHank}\cite{Mei2006:MatValParaprods}\cite{Nazarov-Pisier-Treil-Volberg2002:EstsVecCarlesonEmbThmVecParaprods}\cite{Nazarov-Treil-Volberg1997:CounterExInfDimCarlesonEmbThm}}.
	
	Before we get to the meat of this paper, we define the differentiation operator $D:\Hol\left(\YC\right)\to\Hol\left(\YC\right)$ by $Df\left(z\right)=zf'\left(z\right)+f\left(z\right)$. With respect to the monomial basis, $D$ acts like a diagonal matrix. This presents an elementary way of taking arbitrary powers of $D$: For $\alpha\in\R$, we set
	\[
	D^\alpha f\left(z\right)=\sum_{n=0}^\infty \left(1+n\right)^\alpha \hat f\left(n\right) z^n,\quad z\in\D.
	\]
	Another convenience of working with $D$ in place of ordinary differentiation is that it does not annihilate constants. In fact we can say more: For each $\alpha\in\R$, $D^\alpha:\Hol\left(\YC\right)\to\Hol\left(\YC\right)$ is a bijection that leaves $\OC\left(\YC\right)$ invariant. 
	
	From a technical point of view, the present paper is mainly concerned with $H^2(\HC)$-boundedness of operators of the type $D^\alpha\Gamma_\phi$, with $\alpha>0$ and $\phi\in\Hol\left(\LC\right)$. The present paper was originally motivated by the natural appearance of such operators in control theory, e.g. \cite{Jacob-Rydhe-Wynn2014:WeightWeissConjRKTGenHankOps}. However, they also have implications to our understanding of $\textrm{BMOA}_{\NC\PC}\left(\LC\right)$. Our investigation motivates the definition of a class which we refer to as \emph{Carleson} $\textrm{BMOA}$:
	\begin{definition}\label{def:C}
		Let $\phi\in\Hol\left(\LC\right)$. We then say that $\phi\in \textrm{BMOA}_{\CC}\left(\LC\right)$ if and only if
	\begin{equation}\label{Eq:AAC}
		\left\|\phi\right\|_{\textrm{BMOA}_\CC}^2=\sup_{f\in\OC_1\left(\HC\right)} \int_\D \left\| \left(D\phi\right)\left(z\right) f\left(\conj{ z}\right)\right\|_\HC^2\left(1-\left|z\right|^2\right)\, dA\left(z\right)<\infty.
	\end{equation}
	\end{definition}
	Since $D$ does not annihilate constants, $\left\|\cdot\right\|_{\textrm{BMOA}_\CC}$ is a proper norm, and not a semi-norm. 
	
	\subsection{Main result and corollaries}
	\begin{theorem}\label{Theorem:HankelCarleson}
		Let $\HC$ be a separable Hilbert space, $\LC$ its space of bounded linear transformations. Let $\alpha>0$ and suppose that $\phi:\D\to\LC$ is analytic. Then $D^\alpha \Gamma_\phi$ is $H^2\left(\HC\right)$-bounded if and only if $D^\alpha \phi\in \textrm{BMOA}_{\CC}\left(\LC\right)$, i.e. 
		\[
		\left\|D^\alpha\Gamma_\phi\right\|_{H^2\left(\HC\right)\to\HC^2\left(\HC\right)}=\sup_{f\in\OC_1\left(\HC\right)}\left\|D^\alpha\Gamma_\phi f\right\|_{H^2\left(\HC\right)}<\infty
		\]
		if and only if
		\[
		\left\|D^\alpha\phi\right\|_{\textrm{BMOA}_\CC}=\sup_{f\in\OC_1\left(\HC\right)}\int_\D \left\| \left(D^{1+\alpha}\phi\right)\left(z\right) f\left(\conj{ z}\right)\right\|_\HC^2\left(1-\left|z\right|^2\right)dA\left(z\right)<\infty.
		\]
		Moreover,
		\[
		\left\|D^\alpha\Gamma_\phi\right\|_{H^2\left(\HC\right)\to\HC^2\left(\HC\right)}\approx \left\|D^\alpha\phi\right\|_{\textrm{BMOA}_\CC}.
		\]
	\end{theorem}
	
	Theorem \ref{Theorem:HankelCarleson} generalizes a result by Janson and Peetre \cite{Janson-Peetre1988:Paracomms} who obtained essentially the above characterization in the case where $\HC=\C$. We point out that, in the case where $\phi$ is $\LC$-valued, we are forced to avoid the Schur multiplier techniques used in \cite{Janson-Peetre1988:Paracomms}. This is made evident by the discussion in \cite{Davidson-Paulsen1997:PolBddOps}*{Section 4}.
	
	Operators of the type $D\Gamma_\phi$ received a lot of attention in connection to the so called Halmos problem \cite{Halmos1970:TenProbls}*{Problem 6}: 
	\begin{quote}
		If a Hilbert space operator is similar to a Hilbert space contraction, then it is also polynomially bounded (by von Neumann's inequality). Is the converse true?
	\end{quote} Following the works of many authors \citelist{\cite{Aleksandrov-Peller1996:HankOpsSimToContr}\cite{Bourgain1986:SimProblPolBddOpsHSpace}\cite{Foguel1964:CounterExSz.-NagyProbl}\cite{Paulsen1984:ComplPolBddSimContr}\cite{Peller1982:EstsFcnsPwrBddOpsHSpace}\cite{Sz.-Nagy1959:ComplContOpsUniformlyBddIterates}}, Pisier \cite{Pisier1997:PolBddNotSim} answered this question in the negative. Subsequently, different proofs of the same result have been given in several papers \citelist{\cite{Davidson-Paulsen1997:PolBddOps}\cite{Kislyakov2000:OpsDisSimContr}}. All of these proofs exploit boundedness properties of operators of the type $D\Gamma_\phi$. The following two propositions are essentially from Davidson and Paulsen \cite{Davidson-Paulsen1997:PolBddOps}:
	
	\begin{proposition}\label{Proposition:Davidson-Paulsen1}
		Let $\alpha>0$, and $\HC$ be a separable, infinite-dimensional  Hilbert space, $\LC$ its space of bounded linear transformations. Then there exists an analytic function $\phi:\D\to \LC$ such that $D^\alpha\Gamma_\phi$ is bounded on $H^2\left(\HC\right)$, while $\Gamma_\phi D^\alpha$ is not. Moreover, $\phi$ may be chosen to be rank one-valued.
	\end{proposition}
	
	\begin{proposition}\label{Proposition:Davidson-Paulsen2}
		Let $\alpha>0$, and $\HC$ be a separable, infinite dimensional  Hilbert space, $\LC$ its space of bounded linear transformations. Then there exists a bounded analytic function $\phi:\D\to \LC$ such that $D^\alpha\Gamma_{D^{-\alpha}\phi}$ is not bounded on $H^2\left(\HC\right)$.
	\end{proposition}
	\begin{remark}\label{Remark:Davidson-Paulsen}
		Proposition \ref{Proposition:Davidson-Paulsen1} is stated for $\alpha=1$ in \cite{Davidson-Paulsen1997:PolBddOps}*{Example 4.6}. Proposition \ref{Proposition:Davidson-Paulsen2} is essentially stated for $\alpha=1$ in \cite{Davidson-Paulsen1997:PolBddOps}*{Corollary 4.2}, but this does not explicitly mention the boundedness of $\phi$, even though it follows from the original proof. A dyadic analogue of this result has been proved by Mei \cite{Mei2006:MatValParaprods}. For the convenience of the reader, we present proofs of the above propositions in Section \ref{Sec:Davidson-Paulsen}.
	\end{remark}
	
	Combining the results by Davidson and Paulsen with Theorem \ref{Theorem:HankelCarleson}, we are able to derive several interesting results.
		
	Given $\phi\in\Hol\left(\LC\right)$, we define the function  $\phi^\#:z\mapsto \phi\left(\conj{ z}\right)^*$. This is the function obtained by taking the Hilbert space conjugate of each Taylor coefficient of $\phi$. Note that $\Gamma_\phi D=\left(D\Gamma_{\phi^\#}\right)^*$. By Proposition \ref{Proposition:Davidson-Paulsen1} and Theorem \ref{Theorem:HankelCarleson}, it follows that $\textrm{BMOA}_{\CC}\left(\LC\right)$ is not closed under coefficient conjugation (cf. \cite{Aleman-Perfekt2012:HankFrmsEmbThmsDirichletSpaces}*{Proposition 3.3}). On the other hand, $\textrm{BMOA}_{\NC\PC}\left(\LC\right)$ is obviously closed under coefficient conjugation. We obtain the following corollary:
	\begin{corollary}\label{Corollary:CNP}
		Let $\HC$ be a separable infinite-dimensional Hilbert space, $\LC$ its space of bounded linear transformations. Then $\textrm{BMOA}_\CC\left(\LC\right)$ is not closed under the map $\phi\mapsto\phi^\#$, where $\phi^\#\left(z\right)=\phi\left(\conj{ z}\right)^*$. In particular 
		\[
		\textrm{BMOA}_{\CC}\left(\LC\right)\ne \textrm{BMOA}_{\NC\PC}\left(\LC\right),
		\]
		i.e. $H^2\left(\HC\right)$-boundedness of $\Gamma_\phi$ is not characterized by the anti-analytic Carleson embedding condition indicated by Theorem \ref{Theorem:HankelCarleson}.
	\end{corollary}
	
	Corollary \ref{Corollary:CNP} motivates the following definition:
	\begin{definition}\label{def:CS}
		Let $\phi\in\Hol\left(\LC\right)$. We then say that $\phi\in \textrm{BMOA}_{\CC^\#}\left(\LC\right)$ if and only if $\phi^\#\in \textrm{BMOA}_{\CC}\left(\LC\right)$.
	\end{definition}
			
	Consider now the relation 
	\begin{align}\label{Eq:LeibnizDecomposition}
	\Gamma_{D\phi} 
	&=
	D\Gamma_{\phi}+\left(D\Gamma_{\phi^\#}\right)^*-\Gamma_\phi,
	\end{align}
	which is obtained by duality, and the Leibniz rule for $D$. The operator $\Gamma_\phi$ is bounded on $H^2\left(\HC\right)$, whenever any of the other terms in \eqref{Eq:LeibnizDecomposition} is bounded, since then $D\phi$ is a Bloch function (cf. Lemma \ref{Lemma:HankelCarlesonBloch} below). In the light of Theorem \ref{Theorem:HankelCarleson}, it is then clear from \eqref{Eq:LeibnizDecomposition} that
	\begin{equation}\label{Eq:CC*NP}
	\textrm{BMOA}_{\CC}\left(\LC\right)\cap \textrm{BMOA}_{\CC^\#}\left(\LC\right)\subsetneq \textrm{BMOA}_{\NC\PC}\left(\LC\right).
	\end{equation}
	We point out that the above inclusion also follows implicitly from the proof of \cite{Nazarov-Pisier-Treil-Volberg2002:EstsVecCarlesonEmbThmVecParaprods}*{Theorem 0.8}. However, we obtain also that the inclusion is strict. To see that this is so, suppose that it is not. This would only be possible if $\textrm{BMOA}_{\NC\PC}\left(\LC\right)$ was contained in $\textrm{BMOA}_{\CC}\left(\LC\right)$. By another application of Theorem \ref{Theorem:HankelCarleson}, this would contradict Proposition \ref{Proposition:Davidson-Paulsen2}. We summarize the above discussion:
	\begin{corollary}\label{Corollary:CC*NP}
		Let $\HC$ be a separable Hilbert space, $\LC$ its space of bounded linear transformations. If $\phi:\D\to\LC$ is an analytic function such that
		\[
		\left\|\phi\right\|_{\textrm{BMOA}_\CC}=\sup_{f\in\OC_1\left(\HC\right)}\int_\D \left\| \left(D\phi\right)\left(z\right) f\left(\conj{ z}\right)\right\|_\HC^2\left(1-\left|z\right|^2\right)dA\left(z\right)<\infty,
		\]
		and
		\[
		\left\|\phi^\#\right\|_{\textrm{BMOA}_\CC}=\sup_{f\in\OC_1\left(\HC\right)}\int_\D \left\| \left(D\phi\right)\left(\conj{ z}\right)^* f\left(\conj{ z}\right)\right\|_\HC^2\left(1-\left|z\right|^2\right)dA\left(z\right)<\infty,
		\]
		then
		\[
		\left\|\Gamma_\phi\right\|_{H^2\left(\HC\right)\to\HC^2\left(\HC\right)}=\sup_{f\in\OC_1\left(\HC\right)}\left\|\Gamma_\phi f\right\|_{H^2\left(\HC\right)}<\infty.
		\]
		Moreover,
		\[
		\left\|\Gamma_\phi\right\|_{H^2\left(\HC\right)\to\HC^2\left(\HC\right)}\lesssim \left\|\phi\right\|_{\textrm{BMOA}_\CC}+\left\|\phi^\#\right\|_{\textrm{BMOA}_\CC}.
		\]
		If $\HC$ is infinite dimensional, then the converse statement does not hold.		
	\end{corollary}	
	
	Condition \eqref{Eq:AAC} states that $H^2\left(\HC\right)$ is continuously embedded into $L^2\left(\D,\HC,d\mu\right)$, where $\mu$ is a certain operator valued measure. It is natural to think of this as an embedding of anti-analytic functions, rather than analytic ones. For this reason, we call \eqref{Eq:AAC} the anti-analytic Carleson embedding, to be distinguished from the analytic one, which is given by the straightforward modification \eqref{Eq:AC} below. In the scalar case it is obvious that these two conditions are equivalent. In the general case, this is no longer obvious. In fact, whether or not the two conditions define the same class of functions was posed as an open question by Nazarov, Treil, and Volberg in \cite{Nazarov-Treil-Volberg1997:CounterExInfDimCarlesonEmbThm}. They later restated the question in a joint paper with Pisier \cite{Nazarov-Pisier-Treil-Volberg2002:EstsVecCarlesonEmbThmVecParaprods}. We answer this question in the negative:
	
	\begin{corollary}
		Let $\HC$ be a separable infinite-dimensional Hilbert space, $\LC$ its space of bounded linear transformations. Then there exists a bounded analytic function $\phi:\D\to\LC$ such that
		\begin{equation}\label{Eq:AC}
			\sup_{f\in\OC_1\left(\HC\right)}\int_\D \left\| \left(D\phi\right)\left(z\right) f\left(z\right)\right\|_\HC^2\left(1-\left|z\right|^2\right)dA\left(z\right)<\infty,
		\end{equation}
		while
		\[\tag{\ref{Eq:AAC}$'$}
		\sup_{f\in\OC_1\left(\HC\right)}\int_\D \left\| \left(D\phi\right)\left(z\right) f\left(\conj{ z}\right)\right\|_\HC^2\left(1-\left|z\right|^2\right)dA\left(z\right)=\infty.
		\]
	\end{corollary}
	The proof is as follows: Since $D$ is an isomorphism from $H^2\left(\HC\right)$ to the standard weighted Bergman space $A_1^2\left(\HC\right)$, it follows from the Leibniz rule that \eqref{Eq:AC} is satisfied whenever $\phi$ is bounded. On the other hand, by Proposition \ref{Proposition:Davidson-Paulsen2} and Theorem \ref{Theorem:HankelCarleson}, there exists $\phi\in H^\infty\left(\LC\right)$ that satisfies $\left(\ref{Eq:AAC}'\right)$. \qed
	
	Using standard arguments involving duality, $D^\alpha\Gamma_\phi$ is $H^2\left(\HC\right)$-bounded if and only if $\phi$ is in the dual of the space
	\[
	D^{-\alpha}  \left(  \left(D^\alpha H^2\left(\HC\right) \right)\hat \otimes \conj{H^2\left(\HC\right)} \right)=D^{-\alpha}  \left(  \left(D^\alpha H^2\left(\SC^2\right) \right) \cdot H^2\left(\SC^2\right) \right).
	\]
	A similar statement holds for boundedness of $D^\alpha\Gamma_{\phi^\#}$. This yields an alternative formulation of Theorem \ref{Theorem:HankelCarleson}:
	\begin{corollary}\label{Corollary:Duality}
		Let $\HC$ be a separable Hilbert space, $\LC$ its space of bounded linear transformations. If $\alpha>0$, then $\textrm{BMOA}_{\CC}\left(\LC\right)$ is the dual of
		\[
		D^{-\alpha}  \left(  \left(D^\alpha H^2\left(\HC\right) \right)\hat \otimes \conj{H^2\left(\HC\right)} \right)=D^{-\alpha}  \left(  \left(D^\alpha H^2\left(\SC^2\right) \right) \cdot H^2\left(\SC^2\right) \right),
		\]
		while $\textrm{BMOA}_{\CC^\#}\left(\LC\right)$ is the dual of
		\[
		D^{-\alpha}  \left(H^2\left(\HC\right)\hat \otimes \conj{D^\alpha H^2\left(\HC\right)} \right)=D^{-\alpha}  \left(H^2\left(\SC^2\right) \cdot   \left(D^\alpha H^2\left(\SC^2\right) \right) \right).
		\]
	\end{corollary}
			
	We return for a moment to the scalar case. By the square function characterization of $H^1$, due to Fefferman and Stein \cite{Fefferman-Stein1972:HpSpaces}, it follows that
	\begin{equation}\label{Eq:Inclusion}
	D^{-1}  \left(  \left(D^1 H^2 \right)\cdot H^2 \right)\subseteq H^1.
	\end{equation}
	A generalization to general $\alpha>0$, which also yields equality of function spaces in \eqref{Eq:Inclusion}, has been obtained by Cohn and Verbitsky \cite{Cohn-Verbitsky2000:FactTentSpacesHankOps}. By Corollary \ref{Corollary:Duality}, the dual inclusion becomes
	\[
	\textrm{BMOA}_{\NC\PC}\subseteq \textrm{BMOA}_{\CC}.
	\]
	Combined with Corollary \ref{Corollary:CC*NP}, this implies the well-known result that $\textrm{BMOA}_{\CC}=\left(H^1\right)^*$. For this argument to work, it suffices to use Theorem \ref{Theorem:HankelCarleson} with (say) $\alpha=1$, a special case which is substantially simpler to prove.
	
	The paper is structured as follows: In Section \ref{Sec:Preliminaries} we fix notation, and review some preliminary material. Of particular importance are some Bergman type spaces of analytic functions. In Section \ref{Sec:MainProof} we prove Theorem \ref{Theorem:HankelCarleson}. In Section \ref{Sec:SpecialCases} we discuss and compare the special cases of $\HC$-valued, and $\HC^*$-valued symbols, and point out some significant differences between these. In Section \ref{Sec:Davidson-Paulsen} we provide proofs of the Propositions \ref{Proposition:Davidson-Paulsen1} and \ref{Proposition:Davidson-Paulsen2}.
	
	\section{Preliminaries and further notation}\label{Sec:Preliminaries}
	
	We use the standard notation $\Z$, $\R$, and $\C$ for the respective rings of integers, real numbers, and complex numbers. By $\N$ we denote the set of strictly positive elements of $\Z$, while $\N\cup\left\{0\right\}$ is denoted by $\N_0$.
	
	Given two parametrized sets of nonnegative numbers $\left(A_i\right)_{i\in I}$ and $\left(B_i\right)_{i\in I}$, we use the notation $A_i\lesssim B_i$, $i\in I$ to indicate the existence of a positive constant $C$ such that $\forall i\in I$, $A_i\le CB_i$. We then say that $A_i$ is bounded by $B_i$, and refer to $C$ as a bound. Sometimes we allow ourselves to not mention the index set $I$ and instead let it be implicit from the context. If $A_i\lesssim B_i$ and $B_i\lesssim A_i$, then we write $A_i\approx B_i$. We then say that $A_i$ and $B_i$ are comparable.
	
	The Hilbert space adjoint of $A\in\LC$ is denoted by $A^*$. We sometimes identify $x\in\HC$ with the rank one operator $\C\ni c\mapsto cx\in\HC$. Note that $x^*$ is then the linear functional $\HC\ni y\mapsto \left\langle y,x\right\rangle_\HC\in\C$.
	
	The dual of a Banach space $\YC$ will be denoted by $\YC^*$. With Hilbert spaces in mind, we equip $\YC^*$ with an anti linear structure, rather than the standard linear ditto. Thus, the duality pairing $\left\langle y,y^*\right\rangle_\YC$, of $y\in\YC$ and $y^*\in\YC^*$, becomes anti linear in $y^*$.
	
	We define the tensor product of two elements $x,y\in\HC$ as the rank one operator $x\otimes y:z\mapsto\left\langle z,y\right\rangle_\HC x$. The tensor product is anti linear in its second argument. The projective tensor product $\HC\hat \otimes\HC$, is the closed linear span of $\left\{x\otimes y\right\}_{x,y\in\HC}$, with respect to the norm
	\[
	\left\|T\right\|_{\wedge}=\inf  \left\{\sum_k \left\|x_k\right\|_\HC \left\|y_k\right\|_\HC;T=\sum_k x_k\otimes y_k \right\}.
	\]
	The space $\HC\hat \otimes\HC$ can be isometrically identified with $\SC^1$. The dual of $\SC^1$ is isometrically identified with $\LC$ via the pairing
	\[
	\left\langle T,B\right\rangle_{\SC^1}=\tr \left(B^*T\right) =\sum_{n}\left\langle Te_n,Be_n\right\rangle_\HC=\sum_{k}\left\langle x_k,By_k\right\rangle_\HC,
	\]
	where $B\in\LC$, $\left(e_n\right)_{n=0}^\infty$ is any orthonormal basis of $\HC$, and $\sum_k x_k\otimes y_k$ is any representation of $T$, cf. Wojtaszczyk \cite{Wojtaszczyk1991:BSpacesForAnalysts}*{III.B.26}.
	
	An important property of Hardy spaces $H^p\left(\XC\right)$ is that, given certain properties of $\XC$, $H^p\left(\XC\right)$ may be isometrically identified as a  subspace of $L^p\left(\T,\XC\right)$. The precise statement is as follows:
	\begin{proposition}\label{Proposition:BoundaryIdentification}
		Let $p\in[1,\infty]$, and $f\in H^p\left(\XC\right)$. 
		\begin{itemize}
			\item[$\left(i\right)$] If $\XC\in\left\{\C,\HC,\SC^1\right\}$, then there exists a function $bf\in L^p\left(\T,\XC\right)$ such that for $m$-a.e. $\zeta\in\T$, $\lim_{r\to 0}f_r\left(\zeta\right)=bf\left(\zeta\right)$ in the norm topology on $\XC$. Moreover, $f_r\to bf$ in $L^p\left(\T,\XC\right)$, and
			\[
			\int_{\T}\left(bf\right)\left(\zeta\right)\conj{ \zeta}^ndm\left(\zeta\right)=
			\left\{
			\begin{array}{cl}
			\hat f\left(n\right) & \textnormal{for }n\in\N_0,
			\\
			0 & \textnormal{for }n\notin\N_0,
			\end{array}
			\right.
			\]
			
			\item[$\left(ii\right)$] If $\XC=\LC$, then there exists a function $bf\in L^p\left(\T,\XC\right)$ such that for $m$-a.e. $\zeta\in\T$, $\lim_{r\to 0}f_r\left(\zeta\right)=bf\left(\zeta\right)$ in the strong operator topology. Moreover, $\left\|bf\right\|_{L^p\left(\T,\XC\right)}=\left\|f\right\|_{H^p\left(\XC\right)}$, and all $x,y\in\HC$
			\[
			\int_{\T}\left\langle \left(bf\right)\left(\zeta\right)x,y\right\rangle_\HC\conj{ \zeta}^ndm\left(\zeta\right)=
			\left\{
			\begin{array}{cl}
			\langle \hat f\left(n\right)x,y\rangle_\HC & \textnormal{for }n\in\N_0,
			\\
			0 & \textnormal{for }n\notin\N_0,
			\end{array}
			\right.
			\]
		\end{itemize}
		In particular, we may identify the Taylor coefficients of $f$ with the Fourier coefficients of $bf$.
	\end{proposition}
	In the scalar case, the above result is proved in any serious introduction to Hardy spaces. We mention \cite{Garnett2007:BddAnalFcnsBook}. We refer to \cite{Nikolski2002:EasyReading} for the case $\XC=\HC$, and	 \cite{Rosenblum-Rovnyak1985:HardyClassesOpTheory} for the case $\XC=\LC$. The statement for $\XC=\SC^1$ holds because $\SC^1$ has the so-called analytic Radon--Nikodym property, see \citelist{\cite{Bukhvalov-Danilevich1982:BdryPropsAnalHarmFcnsValBSpace}\cite{Haagerup-Pisier1989:FactAnalFcnsNon-CommL1Spaces}}.
		
	We define the formal duality pairing between $f\in\Hol\left(\YC\right)$ and $g\in\Hol\left(\YC^*\right)$ as
	\[
	\left\langle f,g\right\rangle= \sum_{n=0}^\infty \left\langle \hat f\left(n\right), \hat g\left(n\right)\right\rangle_{\YC}.
	\]
	The pairing is well defined if $f\in\OC\left(\YC\right)$ or $g\in\OC\left(\YC^*\right)$, and generalizes the inner product on $H^2\left(\HC\right)$. Note that $\left\langle D^\alpha f,g\right\rangle = \left\langle f, D^\alpha g\right\rangle$, and, in the case where $\YC=\HC$, $\left\langle f,\Gamma_\phi g\right\rangle=\left\langle f\otimes \conjvar{g},\phi \right\rangle$. 
		
	We will make use of two related notions of weighted Bergman spaces. For $\beta>-1$, we define two finite measures on $\D$:
	\[
	dA_{\beta}\left(z\right)=\frac{1+\beta}{\pi}  \left(1-\left|z\right|^2 \right)^{\beta}dA\left(z\right)\quad\textnormal{and}\quad dA_{\beta,\log}\left(z\right)=\frac{1+\beta}{\pi}  \left(\log  \left(\frac{1}{|z|^2} \right) \right)^{\beta}dA\left(z\right).
	\]
	Here $dA$ denotes area measure on $\C$. For $p\in\left[1,\infty\right)$, we denote by $L_{\beta}^p\left(\YC\right)$ the space of strongly measurable functions $f:\D\to\YC$ such that
	\[
	\left\|f\right\|_{L_{\beta}^p\left(\YC\right)}^p=\int_\D\left\|f\left(z\right)\right\|_\YC^p\, dA_{\beta}\left(z\right)<\infty.
	\]
	We then define the standard weighted Bergman space $A_{\beta}^p\left(\YC\right)=L_{\beta}^p\left(\YC\right)\cap \Hol\left(\YC\right)$.	We similarly define the logarithmically weighted spaces $L_{\beta,\log}^p\left(\YC\right)$ and $A_{\beta,\log}^p\left(\YC\right)$, with $dA_{\beta,\log}$ in place of $dA_{\beta}$. An enlightening reference for standard weighted Bergman spaces with $\YC=\C$ is \cite{Hedenmalm-Korenblum-Zhu2000:BergmanSpacesBook}. We remark that many of the results presented below for $\YC$-valued functions follow by the same proofs as in the scalar case.
	
	The above two notions of Bergman spaces are to a large extent interchangeable:
	\begin{proposition}\label{Proposition:InterchangeableWeights}
		Let $p\in\left[1,\infty\right)$, $\beta>-1$, and $\YC$ be an arbitrary Banach space. We then have that
		\[
		\left\|f\right\|_{A_{\beta,\log}^p\left(\YC\right)} \approx \left\|f\right\|_{A_{\beta}^p\left(\YC\right)},\quad f\in\Hol\left(\YC\right).
		\]
		The corresponding bounds depend on $p$ and $\alpha$.
	\end{proposition}
	One of the above bounds is obtained using the pointwise estimate
	\[
	1-\left|z\right|^2\le \log  \left(\frac{1}{\left|z\right|^2} \right),\quad z\in \D,
	\]
	and the other by using subharmonicity of the function $z\mapsto \left\|f\left(z\right)\right\|_\YC^p$. We refer the interested reader to the easily modified proof of \cite{Garnett2007:BddAnalFcnsBook}*{Lemma VI.3.2} for details. 
	
	A multiplier is an operator $\lambda:\Hol\left(\YC\right)\ni f\mapsto \lambda f\in\Hol\left(\YC\right)$ given by
	\[
	\lambda f\left(z\right)=\sum_{n=0}^\infty \lambda_n\hat f\left(n\right) z^n,\quad z\in\D,
	\]
	for some scalar sequence $\left(\lambda_n\right)_{n=0}^\infty$. With some abuse of the terminology in \cite{Buckley-Koskela-Vukotic1999:FracIntDiffBergmanSpaces}, we say that a multiplier is small if $\left|\lambda_n\right|\lesssim\frac{1}{1+n}$. Using ideas from the proof of \cite{Arregui-Blasco2002:MultplrsVecValBergmanSpaces}*{Theorem 3.2}, one can prove the following result, which we refer to as the small multiplier property for Bergman spaces.
	\begin{proposition}\label{Proposition:SmallMultipliers}
		Let $p\in\left[1,\infty\right)$, $\beta>-1$, and $\YC$ be an arbitrary Banach space. Then small multipliers act boundedly on $A_{\beta}^p\left(\YC\right)$.
	\end{proposition}
		
	The spaces $A_{\beta,\log}^2\left(\HC\right)$ and $A_{\beta}^2\left(\HC\right)$ are closed subspaces of $L_{\beta,\log}^2\left(\HC\right)$ and $L_{\beta}^2\left(\HC\right)$ respectively. The corresponding orthogonal projections are denoted by $P_{\beta,\log}$ and $P_{\beta}$. A calculation shows that if $\phi \in\Hol\left(\LC\right)$ and $f\in\Hol\left(\HC\right)$ are sufficiently regular, then
	\begin{equation}\label{Eq:BergmanHankelFormulalog}
	P_{\beta,\log}\left(\phi \conjvar{f} \right)\left(z\right)
	=
	\sum_{n=0}^\infty   \left(\sum_{m=0}^\infty   \left(\frac{1+n}{1+m+n} \right)^{1+\beta} \hat \phi \left(m+n\right)\hat f\left(m\right) \right)z^n,\quad z\in\D.
	\end{equation}
	and
	\begin{equation}\label{Eq:BergmanHankelFormula}
	P_{\beta}\left(\phi \conjvar{f} \right)\left(z\right)
	=
	\sum_{n=0}^\infty   \left(\sum_{m=0}^\infty \frac{\Gamma\left(1+m+n\right)\Gamma\left(2+\beta+n\right)}{\Gamma\left(2+\beta+m+n\right)\Gamma\left(1+n\right)} \hat \phi \left(m+n\right)\hat f\left(m\right) \right)z^n,\quad z\in\D.
	\end{equation}
	Here $\Gamma:\C\setminus \left\{-1,-2,\ldots\right\}\to\C$ is the standard $\Gamma$-function. By \eqref{Eq:BergmanHankelFormulalog} and \eqref{Eq:BergmanHankelFormula} we are allowed to define $P_{\beta,\log}\left(\phi \conjvar{f} \right)$ and $P_{\beta}\left(\phi \conjvar{f} \right)$ as elements of $\Hol\left(\HC\right)$, whenever $\phi\in\Hol\left(\LC\right)$ and $f\in\OC\left(\HC\right)$. In this sense, they are analogues of \eqref{Eq:HardyHankelFormula}. 
	
	Using Parseval's identity we obtain
	\[
	\left\|f\right\|_{A_{\beta,\log}^2\left(\HC\right)}^2=\Gamma\left(2+\beta\right)\sum_{k=0}^\infty \frac{\|\hat f\left(k\right)\|_\HC^2}{\left(1+k\right)^{1+\beta}},
	\]
	and
	\[
	\left\|f\right\|_{A_{\beta}^2\left(\HC\right)}^2=\sum_{k=0}^\infty \binom{n+1+\beta}{n}^{-1}\|\hat f\left(k\right)\|_\HC^2,
	\]
	where $\binom{n+1+\alpha}{n}=\frac{\Gamma\left(2+\beta+n\right)}{\Gamma\left(2+\beta\right)\Gamma\left(1+n\right)}$ are generalized binomial coefficients.
		
	The Bloch space $\BC\left(\YC\right)$ is the space of functions $f\in\Hol\left(\YC\right)$ such that
	\[
	\left\|f\right\|_{\BC\left(\YC\right)}=\sup_{z\in\D}\, \left(1-\left|z\right|^2\right)\left\|Df\left(z\right)\right\|_\YC<\infty.
	\]
	In the literature the Bloch space is typically defined by finiteness of the quantity 
	\[
	\left\|f\left(0\right)\right\|_\YC+\sup_{z\in\D}\, \left(1-\left|z\right|^2\right)\left\|f'\left(z\right)\right\|_\YC.
	\]
	We leave it as an exercise to show that these definitions are equivalent.
	
	The Bloch space has the simple property that
	\begin{equation}\label{Eq:BlochWBloch}
	\left\|f\right\|_{\BC\left(\YC^*\right)}=\sup_{\substack{y\in\YC\\\left\|y\right\|_\YC=1}}\left\|\left\langle y,f\right\rangle_\YC\right\|_\BC,
	\end{equation}
	as can be seen by interchanging the order of suprema.
	
	The importance of the Bloch space is that $\BC\left(\YC^*\right)$ is isometric to $A_{\beta}^1\left(\YC\right)^*$ via the pairing
	\[
	\left\langle f,g\right\rangle_{A_{\beta}^1\left(\YC\right)} =\lim_{r\uparrow 1}\int_{\D}\left\langle f\left(rz\right),g\left(z\right)\right\rangle_{\YC}\, dA_{\beta}\left(z\right),\quad f\in A_{\beta}^1\left(\YC\right), g\in\BC\left(\YC^*\right).
	\]
	This follows mostly as in \cite{Hedenmalm-Korenblum-Zhu2000:BergmanSpacesBook}. The major difference is that $\BC\left(\YC^*\right)$ is the Bergman projection of a certain class of measures, rather than $L^\infty\left(\D,\YC^*\right)$, see \cite{Arregui-Blasco2003:BergmanBlochSpacesVecVal}.
	
	\section{Proof of Theorem \ref{Theorem:HankelCarleson}}\label{Sec:MainProof}
	Given $\alpha>0$, let $\beta>\max\left\{2,1+\alpha\right\}$ be an auxiliary parameter. To prove Theorem \ref{Theorem:HankelCarleson}, let $\phi\in\Hol\left(\LC\right)$ and define
	\begin{align*}
	\left\|\phi\right\|_{1,\alpha}&=\sup_{f\in \OC_1\left(\HC\right)} \left\|D^\alpha\Gamma_\phi f\right\|_{H^2\left(\HC\right)},
	\\
	\left\|\phi\right\|_{2,\alpha}&=\sup_{f\in \OC_1\left(\HC\right)} \left\|P_{2\beta-1,\log}  \left(\left(D^{\beta+\alpha}\phi\right)\conjvar{ f}  \right)\right\|_{A_{2\beta-1,\log}^2\left(\HC\right)},
	\\
	\left\|\phi\right\|_{3,\alpha}&=\sup_{f\in \OC_1\left(\HC\right)} \left\|P_{1,\log}  \left(\left(D^{1+\alpha}\phi \right)\conjvar{ f} \right)\right\|_{A_{1,\log}^2\left(\HC\right)},
	\\
	\left\|\phi\right\|_{4,\alpha}&=\sup_{f\in \OC_1\left(\HC\right)} \left\|P_{1}  \left(\left(D^{1+\alpha}\phi \right)\conjvar{ f} \right)\right\|_{A_{1,\log}^2\left(\HC\right)},
	\\
	\left\|\phi\right\|_{5,\alpha}&=\sup_{f\in \OC_1\left(\HC\right)} \left\|P_{1}  \left(\left(D^{1+\alpha}\phi \right)\conjvar{ f} \right)\right\|_{A_{1}^2\left(\HC\right)},
	\\
	\left\|\phi\right\|_{6,\alpha}&=\sup_{f\in \OC_1\left(\HC\right)} \left\|\left(D^{1+\alpha} \phi\right)\conjvar{ f}\right\|_{L_{1}^2\left(\HC\right)}.
	\end{align*}
	Theorem \ref{Theorem:HankelCarleson} is the statement that $\left\|\phi \right\|_{1,\alpha}\approx \left\|\phi \right\|_{6,\alpha}$. We will prove that the quantities $\left\|\phi \right\|_{k,\alpha}$, $1\le k\le 6$ are pairwise comparable.
	
	The outline of the proof is as follows. We show in detail that $\left\|\phi \right\|_{1,\alpha}\lesssim \left\|\phi \right\|_{2,\alpha}$. The reverse estimate, as well as the estimates $\left\|\phi \right\|_{2,\alpha}\approx \left\|\phi \right\|_{3,\alpha}$, and $\left\|\phi \right\|_{3,\alpha}\approx \left\|\phi \right\|_{4,\alpha}$ are similar, although the last one is substantially simpler than the preceding ones. The statement that $\left\|\phi \right\|_{4,\alpha}\approx \left\|\phi \right\|_{5,\alpha}$ is just a special case of Proposition \ref{Proposition:InterchangeableWeights}. Furthermore, it is trivial that $\left\|\phi \right\|_{5,\alpha}\le \left\|\phi \right\|_{6,\alpha}$. The reverse of this last estimate follows in a routine manner from the following remarkable result by Aleman and Perfekt \cite{Aleman-Perfekt2012:HankFrmsEmbThmsDirichletSpaces}:
	\begin{lemma}\label{Lemma:AlemanPerfekt}
		There exists a constant $C>0$ such that whenever $\psi\in\Hol\left(\LC\right)$ it holds that
		\begin{align*}
		\sup_{f\in\OC_1\left(\SC^2\right)}\int \left\| \left(D\psi\right) \conjvar{ f}\, \right\|_{\SC^2}^2\, dA_1 \le C \sup_{f,g\in\OC_1\left(\SC^2\right)}  \left|\int \tr  \left(   \left(D\psi \right) \conjvar{ f}   \left( Dg \right)^* \right)\, dA_1 \right|.
		\end{align*}
	\end{lemma}
	
	To prove that $\left\|\phi \right\|_{1,\alpha}\lesssim \left\|\phi \right\|_{2,\alpha}$ we will need some lemmata. The first result gives us some preliminary control of $\phi $.
	\begin{lemma}\label{Lemma:HankelCarlesonBloch}
		For each $\alpha>0$ it holds that
		\[
		\left\|D^\alpha \phi \right\|_{\BC\left(\LC\right)}\lesssim \left\|\phi \right\|_{k,\alpha},\quad \phi \in\Hol\left(\LC\right),\ 1\le k\le 6.
		\]
	\end{lemma}
	\begin{proof}
		We consider only the case $k=1$. The other cases are similar. By \eqref{Eq:BlochWBloch} it suffices to prove that
		\begin{equation}\label{Eq:EstBloch}
		\left|\left\langle x, D^{1+\alpha} \phi \left(w\right)y\right\rangle_\HC\right|\lesssim \frac{\left\|\phi \right\|_{\alpha,1}\left\|x\right\|_\HC\left\|y\right\|_\HC}{1-\left|w\right|^2},\quad w\in\D,\ x,y,\in\HC.
		\end{equation}		
		Given $w\in\D$, $x,y\in\HC$, let
		\[
		f\left(z\right)=\sum_{n=0}^\infty \conj{w}^n  \left(1+n-n  \left(\frac{n}{1+n} \right)^\alpha \right)z^ny,\quad g\left(z\right)=\sum_{n=0}^\infty w^nz^nx,\quad z\in\D.
		\]
		A calculation shows that $1+n-n  \left(\frac{n}{1+n} \right)^\alpha$ is bounded in $n$, and so $\left\|f\right\|_{H^2}\left\|g\right\|_{H^2}\lesssim \frac{1}{1-\left|w\right|^2}$. 
		The definition of $\left\|\phi \right\|_{\alpha,1}$ now yields \eqref{Eq:EstBloch}.
	\end{proof}
	\begin{remark}
		Another proof of Lemma \ref{Lemma:HankelCarlesonBloch} is to use \eqref{Eq:BlochWBloch} together with the (already known) scalar version of Theorem \ref{Theorem:HankelCarleson}. Our approach is chosen so that our results do not depend on the scalar case.
	\end{remark}
		
	The qualitative content of the next lemma is known, and due to Peller \cite{Peller1982:VecHankOps}. See also \cite{Peller2003:HankOpsBook}*{Chapter 6.9}. However, the original proof gives a much worse quantitative dependence on $l$. The proof we present is a bit lengthy, and is postponed to the next subsection.
	\begin{lemma}\label{Lemma:OrderControl}
		For each $\alpha>0$ it holds that
		\[
		\left\|D^\alpha \Gamma_{D^{-\alpha-l}\psi }D^l\right\|_{H^2\left(\HC\right)\to H^2\left(\HC\right)}\le C l \left\|\psi \right\|_{\BC\left(\LC\right)},\quad l\in\N,\ \psi \in\Hol\left(\LC\right).
		\]
	\end{lemma}
	
	We are now ready for the main part of the argument. Given $f\in\OC\left(\HC\right)$, and $\phi \in\Hol\left(\LC\right)$, we use the formulas \eqref{Eq:HardyHankelFormula} and \eqref{Eq:BergmanHankelFormula} to write
	\begin{align*}
	\left\|D^\alpha \Gamma_\phi  f\right\|_{H^2\left(\HC\right)}^2
	&=
	\sum_{n=0}^\infty\left(1+n\right)^{2\alpha}  \left\|\sum_{k=0}^\infty \hat \phi \left(n+k\right)\hat f\left(k\right)\right\|_\HC^2
	\\
	&=
	\sum_{n=0}^\infty\frac{1}{\left(1+n\right)^{2\beta}}  \left\|\sum_{k=0}^\infty   \left(\frac{1+n}{1+n+k} \right)^{\alpha+\beta}   \left(D^{\alpha+\beta} \phi  \right)^{\hat{}}\left(n+k\right)\hat f\left(k\right) \right\|_\HC^2
	\\
	&=
	\left\|P_{\alpha+\beta-1,\log}  \left(\left(D^\beta \psi \right)\conjvar{f}  \right)\right\|_{A^2_{2\beta-1,\log}\left(\HC\right)}^2,
	\end{align*}
	where $\psi =D^\alpha \phi $.
	
	A well known fact about standard weighted Bergman spaces is that there exists many bounded projections from $L_{\gamma}^p$ onto $A_{\gamma}^p$, eg. \cite{Hedenmalm-Korenblum-Zhu2000:BergmanSpacesBook}*{Theorem 1.10}. This inspires us to replace $P_{\alpha+\beta-1,\log}$ with $P_{2\beta-1,\log}$. By the triangle inequality
	\begin{align*}
	&\left\|P_{\alpha+\beta-1,\log}  \left(\left(D^\beta \psi  \right)\conjvar{f} \right)\right\|_{A^2_{2\beta-1,\log}}
	\\
	&\le 	\left\|  \left(P_{\alpha+\beta-1,\log}-P_{2\beta-1,\log} \right)  \left(\left(D^\beta \psi \right)\conjvar{f}  \right)\right\|_{A^2_{2\beta-1,\log}}+	\left\|P_{2\beta-1,\log}  \left(\left(D^\beta \psi \right)\conjvar{f}  \right)\right\|_{A^2_{2\beta-1,\log}}.
	\end{align*}
	We carry out a few manipulations with the Taylor coefficients of $\phi $ and $f$, use the power series expansion at the origin of the function $z\mapsto \left(1-z\right)^{\beta-\alpha}$, and apply Minkowski's inequality to obtain
	\begin{align*}
	&\left\|  \left(P_{\alpha+\beta-1,\log}-P_{2\beta-1,\log} \right)  \left(\left(D^\beta \psi \right)\conjvar{f} \right)\right\|_{A^2_{2\beta-1,\log}}^2
	\\
	&=
	\sum_{n=0}^\infty\frac{1}{\left(1+n\right)^{2\beta}}  \left\|\sum_{k=0}^\infty  \left[  \left(\frac{1+n}{1+n+k} \right)^{\alpha+\beta} -   \left(\frac{1+n}{1+n+k} \right)^{2\beta} \right]\left(D^{\beta} \psi \right)^{\hat{}}\left(n+k\right)\hat f\left(k\right) \right\|_\HC^2
	\\
	&=
	\sum_{n=0}^\infty   \left\|\sum_{k=0}^\infty  \left[  \left(\frac{1+n}{1+n+k} \right)^{\alpha} -   \left(\frac{1+n}{1+n+k} \right)^{\beta} \right]\hat \psi \left(n+k\right)\hat f\left(k\right) \right\|_\HC^2
	\\
	&=
	\sum_{n=0}^\infty   \left\|\sum_{k=0}^\infty   \left(\frac{1+n}{1+n+k} \right)^{\alpha} \left[1 -   \left(1-\frac{k}{1+n+k} \right)^{\beta-\alpha} \right]\hat \psi \left(n+k\right)\hat f\left(k\right) \right\|_\HC^2
	\\
	&=
	\sum_{n=0}^\infty   \left\|\sum_{l=1}^\infty\binom{\beta-\alpha}{l}\left(-1\right)^l\sum_{k=0}^\infty \frac{\left(1+n\right)^\alpha \left(1+k\right)^l}{\left(1+n+k\right)^{\alpha+l}}\hat \psi \left(n+k\right)\left(\frac{k}{1+k}\right)^l\hat f\left(k\right) \right\|_\HC^2
	\\
	&\le
	\left(\sum_{l=1}^\infty  \left|\binom{\beta-\alpha}{l} \right|  \left(\sum_{n=0}^\infty   \left\|\sum_{k=0}^\infty \frac{\left(1+n\right)^\alpha \left(1+k\right)^l}{\left(1+n+k\right)^{\alpha+l}}\hat \psi \left(n+k\right)  \left(\frac{k}{1+k} \right)^l\hat f\left(k\right) \right\|_\HC^2 \right)^{1/2} \right)^2
	\\
	&=
	\left(\sum_{l=1}^\infty  \left|\binom{\beta-\alpha}{l}\right|\left\|D^\alpha\Gamma_{D^{-\alpha-l}}D^l f_l\right\|_{H^2\left(\HC\right)} \right)^2,
	\end{align*}
	where $f_l$ is defined by $\hat f_l\left(k\right)=  \left(\frac{k}{1+k} \right)^l\hat f\left(k\right)$. By Stirling's formula, the binomial coefficients $\binom{\beta-\alpha}{l}$ decay like $\frac{1}{l^{1+\beta-\alpha}}$, and since the map $f\mapsto f_l$ is obviously $H^2\left(\HC\right)$-contractive for each $l$, we use Lemma \ref{Lemma:OrderControl} to conclude that
	\[
	\left\|  \left(P_{\alpha+\beta-1,\log}-P_{2\beta-1,\log} \right)  \left(\left(D^\beta \psi \right)\conjvar{f} \right)\right\|_{A^2_{2\beta-1,\log}}\lesssim \left\|\psi \right\|_{\BC\left(\LC\right)}\left\|f\right\|_{H^2\left(\HC\right)},
	\]
	since $\beta>1+\alpha$. Lemma \ref{Lemma:HankelCarlesonBloch} then implies that
	\[
	\left\|P_{\alpha+\beta-1,\log}  \left(\left(D^\beta \psi \right)\conjvar{f}  \right)\right\|_{A^2_{2\beta-1,\log}} \lesssim \left\|\phi \right\|_{2,\alpha}\left\|f\right\|_{H^2\left(\HC\right)}.
	\]
	This proves that $\left\|\phi \right\|_{1,\alpha}\lesssim \left\|\phi \right\|_{2,\alpha}$. It is straightforward to use the same type of argument to show that $\left\|\phi \right\|_{2,\alpha}\lesssim \left\|\phi \right\|_{1,\alpha}$.
	
	In order to prove that $\left\|\phi \right\|_{2,\alpha}\approx \left\|\phi \right\|_{3,\alpha}$, we note that
	\[
	\left\|P_{2\beta-1,\log}  \left(\left(D^\beta \psi \right)\conjvar{f}  \right)\right\|_{A^2_{2\beta-1,\log}}
	=
	\left\|P_{\beta,\log}  \left(\left(D^1 \psi \right)\conjvar{f}  \right)\right\|_{A^2_{1,\log}}.
	\]
	We repeat the above argument in order to replace $P_{\beta,\log}$ with $P_{1,\log}$. This time instead of $\beta>1+\alpha$, we use that $\beta>2$.	A third application of the argument allows us to replace $P_{1,\log}$ with $P_1$, yielding $\left\|\phi \right\|_{3,\alpha}\approx \left\|\phi \right\|_{4,\alpha}$.
	
	As was pointed out earlier, $\left\|\phi \right\|_{4,\alpha}\approx \left\|\phi \right\|_{5,\alpha}$ is just a special case of Proposition \ref{Proposition:InterchangeableWeights}, while the estimate $\left\|\phi \right\|_{5,\alpha}\le \left\|\phi \right\|_{6,\alpha}$ is trivial. For the reverse inequality, if we identify $\HC$ as a subspace of rank one operators in $\SC^2$, it is obvious that
	\begin{align*}
	\sup_{f\in\OC_1\left(\HC\right)}\int \left\|  \left(D^{1+\alpha}\phi \right) \conjvar{f} \right\|_{\HC}^2\, dA_1 \le \sup_{f\in\OC_1\left(\SC^2\right)}\int \left\| \left(D^{1+\alpha}\phi\right) \conjvar{f} \right\|_{\SC^2}^2\, dA_1.
	\end{align*}
	By a simple argument
	\begin{align*}
	  \left|\int \tr  \left(   \left(D^{1+\alpha}\phi  \right)\conjvar{f}  \left(Dg \right)^* \right)dA_1 \right|
	\le
	\left\|\phi \right\|_{\alpha,5}\left\|f\right\|_{H^2\left(\SC^2\right)}\left\|g\right\|_{H^2\left(\SC^2\right)}
	\end{align*}
	holds whenever $f,g\in\OC\left(\SC^2\right)$. By Lemma \ref{Lemma:AlemanPerfekt}, $\left\|\phi \right\|_{6,\alpha}\lesssim\left\|\phi \right\|_{5,\alpha}$. This completes the proof of Theorem \ref{Theorem:HankelCarleson}.
			
	\subsection{Proof of Lemma \ref{Lemma:OrderControl}}
	For $\alpha>0$ we define the operator $\tilde D^\alpha:\Hol\left(\YC\right)\to\Hol\left(\YC\right)$ by
	\[
	\tilde D^\alpha f\left(z\right)=\sum_{n=0}^\infty \frac{\Gamma\left(1+n+\alpha\right)}{\Gamma\left(1+n\right)}\hat f\left(n\right) z^n,\quad z\in\D.
	\]
	A calculation shows that
	\begin{align*}
	\left\langle \tilde D^\alpha f,\psi \right\rangle_{A_{\alpha-1}^1\left(\YC\right)}=
	\Gamma\left(1+\alpha\right)\left\langle f,\psi \right\rangle,
	\end{align*}
	whenever $f\in\OC\left(\YC\right)$ and $\psi\in\BC\left(\YC^*\right)$. 
	
	Going to the case where $\psi\in\BC\left(\LC\right)$, $f,g\in\OC\left(\HC\right)$, we obtain that
	\begin{align*}
	\left\langle f,D^\alpha\Gamma_{D^{-\alpha-l}\psi }D^lg\right\rangle
	&=
	  \left\langle  D^{-\alpha-l}  \left(  \left(D^\alpha f \right)\otimes \conjvar{  \left(D^l g \right)} \right),\psi  \right\rangle
	\\
	&=
	\frac{1}{\Gamma\left(1+\alpha\right)}  \left\langle \tilde D^\alpha D^{-\alpha}D^{-l}  \left(  \left(D^\alpha f \right)\otimes \conjvar{  \left(D^l g \right)} \right),\psi  \right\rangle_{A_{\alpha-1}^1\left(\YC\right)}.
	\end{align*}
	Since $\psi \in\BC\left(\LC\right)$, we have that 
	\[
	\left|\left\langle f,D^\alpha\Gamma_{D^{-\alpha-l}\psi }D^lg\right\rangle\right| \lesssim\left\|\psi \right\|_\BC\left\|\tilde D^\alpha D^{-\alpha}D^{-l}  \left(  \left(D^\alpha f \right)\otimes \conjvar{  \left(D^l g \right)} \right)\right\|_{A_{\alpha-1}^1\left(\SC^1\right)}.
	\]
	Following the ideas in \cite{Buckley-Koskela-Vukotic1999:FracIntDiffBergmanSpaces}, we use Stirling's formula to see that $\tilde D^\alpha D^{-\alpha}$ acts like the identity plus a small multiplier. By Propositions \ref{Proposition:InterchangeableWeights} and \ref{Proposition:SmallMultipliers}, we can now complete the proof of Lemma \ref{Lemma:OrderControl} by showing that
	\[
	\left\| D^{-l}  \left(  \left(D^\alpha f \right)\otimes \conjvar{  \left(D^l g \right)} \right)\right\|_{A_{\alpha-1,\log}^1\left(\SC^1\right)}\lesssim l \left\|f\right\|_{H^2\left(\HC\right)}\left\|g\right\|_{H^2\left(\HC\right)}.
	\]
	
	First we perform a simple decomposition of $f$ and $g$ into low and high frequencies. Assume that $f$ and $g$ are of degree at most $l$. By the triangle inequality we have
	\begin{align*}
	&\left\|D^{-l}  \left(  \left(D^\alpha f \right)\otimes \conjvar{  \left(D^l g \right)} \right)\right\|_{A_{\alpha-1,\log}^1\left(\SC^1\right)}
	\\
	&\le 
	\sum_{m,n=0}^l\left\|\hat f\left(m\right)\right\|_{\HC}\left\|\hat g\left(n\right)\right\|_{\HC} \left\| D^{-l}\left(\left(D^\alpha z^m\right)\left(D^l z^n\right)\right)\right\|_{A_{\alpha-1,\log}^1}
	\\
	&=
	\sum_{m,n=0}^l\frac{\left(1+m\right)^\alpha\left\|\hat f\left(m\right)\right\|_{\HC}\left(1+n\right)^l\left\|\hat g\left(n\right)\right\|_{\HC}}{\left(1+m+n\right)^l} \left\| z^{m+n}\right\|_{A_{\alpha-1,\log}^1}.
	\end{align*}
	Using polar coordinates we compute that
	\[
	\left\| z^{m+n}\right\|_{A_{\alpha-1,\log}^1}=\frac{2^\alpha\Gamma\left(1+\alpha\right)}{\left(2+m+n\right)^\alpha},
	\]
	and so
	\begin{align*}
	\left\|D^{-l}  \left(  \left(D^\alpha f \right)\otimes \conjvar{  \left(D^l g \right)} \right)\right\|_{A_{\alpha-1,\log}^1\left(\SC^1\right)}
	&\lesssim \sum_{m,n=0}^l\left\|\hat f\left(m\right)\right\|_{\HC}\left\|\hat g\left(n\right)\right\|_{\HC}
	\\
	&\le l\left\|f\right\|_{H^2\left(\HC\right)}\left\|g\right\|_{H^2\left(\HC\right)},
	\end{align*}
	by Cauchy--Schwarz's inequality. Thus the low frequencies exhibit the desired behaviour.
	
	We now consider the high frequencies. Assume that $  \left(D^\alpha f \right)\otimes \conjvar{  \left(D^l g \right)}$ has a zero of order $l$ at the origin. We can then use Lemma \ref{Lemma:PrimitiveNorm}, followed by Cauchy-Schwarz's inequality, and Parseval's identity to obtain that
	\begin{align*}
	&\left\|D^{-l}  \left(  \left(D^\alpha f \right)\otimes \conjvar{  \left(D^l g \right)} \right)\right\|_{A_{\alpha-1,\log}^1\left(\SC^1\right)}
	\\
	&\le
	\frac{\Gamma\left(1+\alpha\right)}{2^l\Gamma\left(1+\alpha+l\right)}  \left(\frac{2+l}{1+l} \right)^l\left\|  \left(D^\alpha f \right)\otimes \conjvar{  \left(D^l g \right)}\right\|_{A_{\alpha+l-1,\log}^1\left(\SC^1\right)}
	\\
	&\le
	\frac{\Gamma\left(\alpha\right)}{2^l\Gamma\left(\alpha+l\right)2l}  \left(\frac{2+l}{1+l} \right)^l\left\|D^\alpha f\right\|_{A_{2\alpha-1,\log}^2\left(\HC\right)}\left\|D^lg\right\|_{A_{2l-1,\log}^2\left(\HC\right)}
	\\
	&=
	\frac{\Gamma\left(\alpha\right)\Gamma\left(1+\alpha\right)^{1/2}\Gamma\left(2l\right)^{1/2}}{2^l\Gamma\left(\alpha+l\right)}  \left(\frac{2+l}{1+l} \right)^l\left\|f\right\|_{H^2\left(\HC\right)}\left\| g\right\|_{H^2\left(\HC\right)}
	\\
	&\lesssim
	l^{1/4-\alpha}\left\|f\right\|_{H^2\left(\HC\right)}\left\| g\right\|_{H^2\left(\HC\right)},
	\end{align*}
	where in the last step, we have used Stirling's formula. Assuming Lemma \ref{Lemma:PrimitiveNorm}, this completes the proof of Lemma \ref{Lemma:OrderControl}.
		
	\begin{lemma}\label{Lemma:PrimitiveNorm}
		Let $\alpha>0$, $N\in\N_0$, and assume that $h\in\Hol\left(\YC\right)$ has a zero of order $N$ at the origin. Then
		\[
		\left\|D^{-l}h\right\|_{A_{\alpha-1,\log}^1\left(\YC\right)}\le \frac{\Gamma\left(1+\alpha\right)}{2^l\Gamma\left(1+\alpha+l\right)}  \left(\frac{2+N}{1+N} \right)^l\left\|h\right\|_{A_{\alpha+l-1,\log}^1\left(\YC\right)},
		\]
		whenever $l\in\N$.
	\end{lemma}
	\begin{proof}
		We will use an idea of Flett \cite{Flett1972:DualIneqHardyLittlewood}. Term by term integration of the power series of $h$ shows that
		\[
		\left(D^{-l} \right)h\left(r\zeta\right)=\frac{1}{\Gamma\left(k\right)r}\int_{s=0}^{r}h_s\left(\zeta\right)  \left(\log  \left(\frac{r}{s} \right) \right)^{l-1}ds,\quad r\in\left[0,1\right),\zeta\in\T.
		\]
		By the triangle inequality
		\begin{align*}
			&\left\|D^{-l}h\right\|_{A_{\alpha-1,\log}^1\left(\YC\right)}
			\\
			&\le \frac{2\alpha}{\Gamma\left(l\right)}\int_{r=0}^1\int_{\T}\int_{s=0}^r\left\|h_s\left(\zeta\right)\right\|_{\YC}  \left(\log  \left(\frac{r}{s} \right) \right)^{l-1}  \left(\log  \left(\frac{1}{r^2} \right) \right)^{\alpha-1}ds\, dm\left(\zeta\right) dr
			\\
			&=
			\frac{\alpha 2^{\alpha}}{\Gamma\left(l\right)}\int_{s=0}^1\int_{\T}\left\|h_s\right\|_{\YC}\, dm \int_{r=s}^1  \left(\log  \left(\frac{1}{s} \right)-\log  \left(\frac{1}{r} \right) \right)^{l-1}  \left(\log  \left(\frac{1}{r} \right) \right)^{\alpha-1}dr\, ds.
		\end{align*}
		By the change of variables $\log\left(\frac{1}{r}\right)/\log\left(\frac{1}{s}\right)=x$ we obtain
		\begin{multline*}
		\int_{r=s}^1  \left(\log  \left(\frac{1}{s} \right)-\log  \left(\frac{1}{r} \right) \right)^{l-1}  \left(\log  \left(\frac{1}{r} \right) \right)^{\alpha-1}dr
		\\
		=\left(\log  \left(\frac{1}{s} \right) \right)^{\alpha+l-1}\int_{x=0}^1\left(1-x\right)^{l-1}x^{\alpha-1}s^xdx.
		\end{multline*}
		
		Therefore
		\begin{align*}
		&\left\|D^{-l}h\right\|_{A_{\alpha-1,\log}^1}
		\\
		&\le 
		\frac{\alpha 2^{\alpha}}{\Gamma\left(l\right)}\int_{s=0}^1\int_{\T}\left\|h_s\right\|_{\YC}\, dm   \left(\log  \left(\frac{1}{s} \right) \right)^{\alpha+l-1} \int_{x=0}^1\left(1-x\right)^{l-1}x^{\alpha-1}s^xdx\, ds
		\\
		&=
		\frac{\alpha 2^{\alpha}}{\Gamma\left(l\right)}\int_{x=0}^1\left(1-x\right)^{l-1}x^{\alpha-1}\int_{s=0}^1\int_{\T}\left\|h_s\right\|_{\YC}\, dm   \left(\log  \left(\frac{1}{s} \right) \right)^{\alpha+l-1} s^xds\, dx.
		\end{align*}
		We now replace the variable $s$ with $s^\delta$, where $\delta=\delta\left(x\right)$ will soon be chosen.
		\begin{align*}
		&\frac{\Gamma\left(l\right)}{\alpha 2^{\alpha}}\left\|D^{-l}h\right\|_{A_{\alpha-1,\log}^1}
		\\
		&\le 
		\int_{x=0}^1\left(1-x\right)^{l-1}x^{\alpha-1}\delta^{\alpha+l}\int_{s=0}^1\int_{\T}\left\|h_{s^\delta}\right\|_{\YC}\, dm   \left(\log  \left(\frac{1}{s} \right) \right)^{\alpha+l-1} s^{\left(1+x\right)\delta-1}ds\, dx
		\\
		&=
		\int_{x=0}^1\left(1-x\right)^{l-1}x^{\alpha-1}\delta^{\alpha+l}\int_{s=0}^1\int_{\T}\frac{\left\|h_{s^\delta}\right\|_{\YC}}{s^{\delta N}}\, dm   \left(\log  \left(\frac{1}{s} \right) \right)^{\alpha+l-1} s^{\left(1+x+N\right)\delta-1}ds\, dx.
		\end{align*}
		Choose $\delta=\frac{2+N}{1+N+x}$. Note that $\delta\ge 1$ whenever $x\in[0,1]$. By assumption, the function $z\mapsto \frac{f\left(z\right)}{z^N}$ is analytic. It follows by subharmonicity that
		\[
		\int_{\T}\frac{\left\|h_{s^\delta}\right\|_{\YC}}{s^{\delta N}}\, dm\le \int_{\T}\frac{\left\|h_{s}\right\|_{\YC}}{s^{N}}\, dm,
		\]
		and so
		\begin{align*}
		&	\frac{\Gamma\left(l\right)}{\alpha 2^{\alpha}}\left\|D^{-l}h\right\|_{A_{\alpha-1,\log}^1}
		\\
		&\le 
		\int_{x=0}^1\left(1-x\right)^{l-1}x^{\alpha-1}\delta^{\alpha+l}\int_{s=0}^1\int_{\T}\left\|h_s\right\|_{\YC}\, dm   \left(\log  \left(\frac{1}{s} \right) \right)^{\alpha+l-1} s^{\left(1+x+N\right)\delta-1-N}ds\, dx
		\\
		&=
		\int_{x=0}^1\left(1-x\right)^{l-1}x^{\alpha-1}  \left(\frac{2+N}{1+N+x} \right)^{\alpha+l}\int_{s=0}^1\int_{\T}\left\|h_s\right\|_{\YC}\, dm   \left(\log  \left(\frac{1}{s} \right) \right)^{\alpha+l-1} s\, ds\, dx
		\\
		&=
		\frac{1 }{2^{l+\alpha}\left(l+\alpha\right)}\left\|h\right\|_{A_{\alpha+l-1,\log}^1}\int_{x=0}^1\left(1-x\right)^{l-1}x^{\alpha-1}  \left(\frac{2+N}{1+N+x} \right)^{\alpha+l}dx.
		\end{align*}
		Replacing the variable $x$ with $\frac{\left(N+1\right)x}{N+2-x}$ we obtain
		\begin{align*}
		\int_{x=0}^1\left(1-x\right)^{l-1}x^{\alpha-1}  \left(\frac{2+N}{1+N+x} \right)^{\alpha+l}dx
		&=  \left(\frac{2+N}{1+N} \right)^l\int_{x=0}^1\left(1-x\right)^{l-1}x^{\alpha-1}dx
		\\
		&=
		  \left(\frac{2+N}{1+N} \right)^l\frac{\Gamma\left(l\right)\Gamma\left(\alpha\right)}{\Gamma\left(l+\alpha\right)},
		\end{align*}
		and the proof of Lemma \ref{Lemma:PrimitiveNorm} is complete.
	\end{proof}
	
	\begin{remark}
		The bound in Lemma \ref{Lemma:PrimitiveNorm} is sharp, as is seen by testing on the function $h\left(z\right)=z^N$. In particular we have that
		\[
		\left\|D^{-l}\right\|_{A_{\alpha-1,\log}^1\left(\YC\right)\to A_{\alpha+l-1,\log}^1\left(\YC\right)}= \frac{\Gamma\left(1+\alpha\right)}{\Gamma\left(1+\alpha+l\right)}.
		\]
		This shows that without the separation of $f$ and $g$ into low and high frequencies, the estimate obtained in Lemma \ref{Lemma:OrderControl} would instead be
		\[
		\left\|D^\alpha \Gamma_{D^{-\alpha-l}\psi }D^l\right\|_{H^2\left(\HC\right)\to H^2\left(\HC\right)}\lesssim 2^l \left\|\psi \right\|_{\BC\left(\LC\right)},\quad \psi\in\Hol\left(\LC\right),
		\]
		which is of course far from sufficient for proving Theorem \ref{Theorem:HankelCarleson}. Still, some of the estimates in the proof of Lemma \ref{Lemma:OrderControl} are very crude, indicating room for improvement. If Lemma \ref{Lemma:OrderControl} could be improved so that for each $l\in\N$
		\[
		\left\|D^\alpha \Gamma_{D^{-\alpha-l}\psi }D^l\right\|_{H^2\left(\HC\right)\to H^2\left(\HC\right)}\le C_l \left\|\psi \right\|_{\BC\left(\LC\right)},
		\]
		where $\sum_{l=1}^\infty \frac{C_l}{l^\gamma}<\infty$ whenever $\gamma>1$, then in the proof of Theorem \ref{Theorem:HankelCarleson} one could immediately prove that $\left\|\phi \right\|_{1,\alpha}\approx\left\|\phi \right\|_{3,\alpha}$, instead of using two iterations of the same argument.
	\end{remark}
	
	\section{$\HC$- and $\HC^*$-valued symbols}\label{Sec:SpecialCases}
	A function $k_w$, where $w\in\D$, defined by
	\[
	k_w\left(z\right)=\frac{1}{1-\conj{w} z},\quad z\in\D,
	\]
	is called a reproducing kernel function for $H^2$. By Parseval's formula, $\left\langle f,k_w\right\rangle=f\left(w\right)$ whenever $f\in H^2$, and $\left\|k_w\right\|_{H^2}^2=\frac{1}{1-\left|w\right|^2}$.	From \citelist{\cite{Blasco1997:VecValBMOAGeomBSpaces}\cite{Bonsall1984:BddnessHankMat}} we gather the following result:
	\begin{proposition}\label{Proposition:H-Valued}
		If $\phi:\D\to\HC$ is analytic, then $\phi \in H^1\left(\HC\right)^*$ if and only if either of the following conditions hold:
		\begin{itemize}
			\item[$\left(i\right)$]
			\[
			\sup_{I\subset\T}\frac{1}{m\left(I\right)}\int_I\left\|b\phi-\left(b\phi\right)_I\right\|_\HC\, dm<\infty.
			\]
			\item[$\left(ii\right)$]
			\[
			\sup_{f\in\OC_1} \int_\D \left|f\left(z\right)\right|^2\left\| \left(D\phi\right)\left(z\right)\right\|_\HC^2\left(1-\left|z\right|^2\right)dA\left(z\right)<\infty.
			\]
			\item[$\left(iii\right)$]
			\[
			\sup_{w\in\D} \left(1-\left|w\right|^2\right)\int_\D \left|k_w\left(z\right)\right|^2\left\| \left(D\phi\right)\left(z\right)\right\|_\HC^2\left(1-\left|z\right|^2\right)dA\left(z\right)<\infty.
			\]
			\item[$\left(iv\right)$]
			\[
			\sup_{f\in\OC_1} \left\|\Gamma_\phi f\right\|_{H^2\left(\HC\right)}<\infty.
			\]
			\item[$\left(v\right)$]
			\[
			\sup_{w\in\D} \left(1-\left|w\right|^2\right)\left\|\Gamma_\phi k_w\right\|_{H^2\left(\HC\right)}<\infty.
			\]
		\end{itemize}
		Moreover, the corresponding norms are comparable.
	\end{proposition}
	We point out that even though the results that $\left(iii\right)\Rightarrow\left(ii\right)$ and $\left(v\right)\Rightarrow\left(iv\right)$ look similar, the relation between them is not trivial. The fact that boundedness of a Hankel operator may be determined by its action on reproducing kernels is often referred to as Bonsall's theorem, and is an example of a so called reproducing kernel thesis. It was shown in \cite{Jacob-Rydhe-Wynn2014:WeightWeissConjRKTGenHankOps} that for scalar-valued symbols, the operators $D^\alpha\Gamma_\phi :H^2\to H^2$ ($\alpha\ge 0$) have a reproducing kernel thesis, while $\left(D^\alpha \Gamma_{\phi^\#}\right)^*:H^2\to H^2$ do not. For $\HC$-valued symbols, $D^\alpha\Gamma_\phi :H^2\to H^2\left(\HC\right)$ ($\alpha\ge 0$) satisfies a reproducing kernel thesis. The proof is the same as in the scalar case. In this section, we investigate the corresponding results for Carleson embeddings.
			
	By specializing Theorem \ref{Theorem:HankelCarleson} to the case of rank one-valued symbols, we obtain the following corollary:
	\begin{corollary}\label{Corollary:HankelCarlesonrankone}
		Let $\alpha>0$ and $\phi\in\Hol\left(\HC\right)$. Then the operator $D^\alpha\Gamma_\phi:H^2\to H^2\left(\HC\right)$ is bounded if and only if 
		\begin{equation}\label{Eq:RankOneCarlesonCondition}
		\sup_{f\in\OC_1}\int_ \D \left|f\left(z\right)\right|^2\left\|D^{1+\alpha}\phi\left(z\right)\right\|_\HC^2\left(1-\left|z\right|^2\right)dA\left(z\right)<\infty.
		\end{equation}
		Moreover, the above supremum is comparable to $\left\|D^\alpha\Gamma_\phi\right\|_{H^2\to H^2\left(\HC\right)}^2$.
	\end{corollary}
	Combined with Proposition \ref{Proposition:H-Valued}, Corollary \ref{Corollary:HankelCarlesonrankone} says that $D^\alpha\Gamma_\phi:H^2\to H^2\left(\HC\right)$ is bounded if and only if $D^\alpha \phi\in \textrm{BMOA}\left(\HC\right)$, i.e. $\phi\in D^{-\alpha} \textrm{BMOA}\left(\HC\right)=\left(D^\alpha H^1\left(\HC\right)\right)^*$. This shows that Corollary \ref{Corollary:HankelCarlesonrankone} could also have been obtained from the factorization $D^\alpha H^1\left(\HC\right)=H^2\cdot D^\alpha H^2\left(\HC\right)$, see \citelist{\cite{Cohn-Verbitsky2000:FactTentSpacesHankOps}\cite{Rydhe2016:CharTriebel-LizorkinSpaces}}.
	
	We now state the corresponding result for functional-valued symbols:
	\begin{corollary}
		Let $\alpha>0$ and $\phi\in\Hol\left(\HC\right)$. Then the operator $D^\alpha\Gamma_{\phi^\#}:H^2\left(\HC\right)\to H^2$ is bounded if and only if
		\begin{equation}\label{Eq:CoRankOneCarlesonCondition}
		\sup_{f\in\OC_1\left(\HC\right)}\int_ \D \left|\left\langle f\left(z\right),\left(D^{1+\alpha}\phi\right)\left(z\right)\right\rangle_\HC\right|^2\left(1-\left|z\right|^2\right)dA\left(z\right)< \infty.
		\end{equation}
		Moreover, the above supremum is comparable to $\left\|D^\alpha\Gamma_{\phi^\#}\right\|_{H^2\left(\HC\right)\to H^2}^2$.
	\end{corollary}
	
	Even though $\HC$ and $\HC^*$ are isomorphic, condition \eqref{Eq:CoRankOneCarlesonCondition} is far more subtle than \eqref{Eq:RankOneCarlesonCondition}. It is easy to show that \eqref{Eq:RankOneCarlesonCondition} implies \eqref{Eq:CoRankOneCarlesonCondition}. The reverse implication does not hold, as is seen by Theorem \ref{Theorem:HankelCarleson} together with Proposition \ref{Proposition:Davidson-Paulsen1}. This also shows that $D^\alpha H^1\left(\HC\right)\ne H^2\left(\HC\right)\cdot D^\alpha H^2$. 
	
	Motivated by Proposition \ref{Proposition:H-Valued}, it is natural to consider the condition
	\begin{equation}\label{Eq:WeakBMOA}
	\sup_{\substack{w\in\D,\\ x\in\HC, \left\|x\right\|_\HC=1}}\left(1-\left|w\right|^2\right)\int_ \D \left|\left\langle k_w\left(z\right)x,\left(D^{1+\alpha}\phi\right)\left(z\right)\right\rangle_\HC\right|^2\left(1-\left|z\right|^2\right)dA\left(z\right)< \infty.
	\end{equation}
	This weak type condition means that the functions $z\mapsto \left\langle \phi\left(z\right),x\right\rangle_\HC$ are in scalar-valued $\textrm{BMOA}$, uniformly for all $x$ in the unit ball of $\HC$.	We use the conditions \eqref{Eq:RankOneCarlesonCondition}, \eqref{Eq:CoRankOneCarlesonCondition}, and \eqref{Eq:WeakBMOA} to define the respective spaces $\textrm{BMOA}_{\CC}\left(\HC\right)$, $\textrm{BMOA}_{\CC^\#}\left(\HC\right)$, and $\textrm{BMOA}_{\WC}\left(\HC\right)$. We then have the strict inclusions 
	\[
	\textrm{BMOA}_{\CC}\left(\HC\right)\subsetneq \textrm{BMOA}_{\CC^\#}\left(\HC\right) \subsetneq \textrm{BMOA}_{\WC}\left(\HC\right).
	\]
	We refer to \cite{Rydhe2016:CounterExsCarlesonEmbThm} for an example showing that the last inclusion is strict.
	
	\section{The Davidson--Paulsen results}\label{Sec:Davidson-Paulsen}
	We will now present the proofs of Propositions \ref{Proposition:Davidson-Paulsen1} and \ref{Proposition:Davidson-Paulsen2}. We once again point out that these are (at most) straightforward adaptations of the arguments used in \cite{Davidson-Paulsen1997:PolBddOps}. It will be convenient to identify $H^2\left(\HC\right)$ with $l^2\left(\N_0,\HC\right)$, and let $\HC= l^2\left(\N_0\right)$. We let $\left(e_n\right)_{n= 0}^\infty$ denote the canonical basis for $l^2\left(\N_0\right)$.
	
	\subsection{Proof of Proposition \ref{Proposition:Davidson-Paulsen1}}
	Let $x\in\HC$ be a fixed vector of unit length, and consider the function $\phi:z\mapsto \sum_{n=0}^\infty \beta_n x\otimes e_n z^n$, where $\left(\beta_n\right)_{n=0}^\infty$ is some scalar sequence of moderate growth. The function $\phi$ is obviously rank one-valued, and with the right choice of $\left(\beta_n\right)_{n=0}^\infty$ it has the property that $D^\alpha\Gamma_\phi$ is bounded on $H^2\left(\HC\right)$, while $\Gamma_\phi D^\alpha$ is not. 
	
	Since the contraction $H^2\left(\HC\right)\ni f\mapsto \left\langle f,x\right\rangle_\HC\in H^2$ maps a subset of the unit sphere in $H^2\left(\HC\right)$ onto the unit sphere in $H^2$, we may instead consider boundedness of $D^\alpha\Gamma_\psi:H^2\left(\HC\right)\to H^2$, where $\psi$ is the $\HC^*$-valued function $z\mapsto \sum_{n=0}^\infty \beta_n e_n^* z^n$.
	
	It will be simpler to consider boundedness of the operators $\left(D^\alpha\Gamma_\psi\right)^*=\Gamma_{\psi^\#}D^\alpha$ and $\left(\Gamma_\psi D^\alpha\right)^*=D^\alpha \Gamma_{\psi^\#}$. Let $X=[\beta_{m+n}e_{m+n}]_{m,n\ge 0}$ be the matrix representation of $\Gamma_{\psi^\#}$. The goal is now to show that $X\  \diag\left(\left(1+n\right)^{\alpha}\right)_{n\ge 0}$ is bounded from $l^2\left(\N_0\right)$ to $l^2\left(\N_0,\HC\right)$, while $\diag\left(\left(1+n\right)^{\alpha}\right)_{n\ge 0} X$ is not. 
	
	Obviously, the operator norm of $D^\alpha X$ is at least as big as the $l^2\left(\N_0,\HC\right)$-norm of each column of the matrix, i.e.
	\[
	\left\|D^\alpha X\right\|_{l^2\left(\N_0\right)\to l^2\left(\N_0,\HC\right)}^2\ge \sup_{k\in\N_0}\sum_{n=0}^\infty \left(1+n\right)^{2\alpha}\left|\beta_{n+k}\right|^2,
	\]
	so if for example $\sum_{n=0}^\infty \left(1+n\right)^{2\alpha}\left|\beta_{n}\right|^2=\infty$, then $D^\alpha X$ is unbounded. On the other hand,
	\[
	\left\langle Xe_n,X e_m\right\rangle_\HC=  \left\{
	\begin{array}{cl}
	\gamma_n^2:=\sum_{k\ge n}\left|\beta_n\right|^2&\textnormal{for $m=n$},\\
	0&\textnormal{otherwise}.
	\end{array} \right.
	\]
	If follows that $  \left(XD^\alpha \right)^*XD^\alpha=\diag  \left(\left(1+n\right)^{2\alpha}\gamma_n^2 \right)_{n\ge 0}$, and so
	\[
	\left\|X D^\alpha\right\|_{l^2\left(\N_0\right)\to l^2\left(\N_0,\HC\right)}^2 = \sup_{n\in\N_0}\left(1+n\right)^{2\alpha}\sum_{k\ge n}\left|\beta_n\right|^2.
	\]
	Now chose $\beta_n=\frac{1}{\left(1+n\right)^{\alpha+1/2}}$ to complete the proof.

	\subsection{Proof of Proposition \ref{Proposition:Davidson-Paulsen2}}
	Given matrices $A=[a_{mn}]_{m,n\ge 0}$ and $B=[b_{mn}]_{m,n\ge 0}$, we define the Schur product $A\star B= \left[a_{mn}b_{mn}\right]_{m,n\ge 0}$. For a fixed matrix $B$, the operator $S_B:A\mapsto A\star B$ is called a Schur multiplier. The Grothendieck-Haagerup criterion, e.g. \cite{Paulsen2002:ComplBddMapsOpAlgs}*{Corollary 8.8}, states that $S_B:\LC\left(\HC\right)\to\LC\left(\HC\right)$ is bounded if and only if there exists sequences $\left(x_n\right)_{n\ge 0}$, $\left(y_n\right)_{n\ge 0}$ in the unit ball of $\HC$ such that $b_{mn}=\left\langle x_n,y_m\right\rangle_\HC$. From this follows the so called Bennett criterion, stating that if $S_B$ is a bounded Schur multiplier, and the iterated limits $\lim_{m\to\infty}\lim_{n\to\infty}b_{mn}$ and $\lim_{n\to\infty}\lim_{m\to\infty}b_{mn}$ both exist, then the limits are equal.
	
	Define an isometry $V:l^2\left(\N_0\right)\to H^2\left(\HC\right)$ by $Ve_n=e_n z^n$, and let $\left(E_{mn}\right)_{m,n\ge 0}$ be the scalar matrices defined by $\left\langle E_{mn}e_l,e_k\right\rangle_\HC=\delta_{mk}\delta_{nl}$. Given a scalar matrix $A=[a_{mn}]_{m,n\ge 0}$, we define the matrices $A_n=\sum_{k+l=n}a_{kl}E_{kl}$, and the function 
	\[
	\phi\left(z\right)=\sum_{n=0}^\infty A_nz^n=\diag  \left(z^k \right)_{k\ge 0}A\ \diag  \left(z^l \right)_{l\ge 0}.
	\]
	From the above relations, $\left\|\phi\right\|_{H^\infty\left(\LC\right)}=\left\|A\right\|_\LC$. Now $\Gamma_\phi$ corresponds to the (operator-valued) Hankel matrix $X=[A_{k+l}]_{k,l\ge 0}$. A calculation shows that
	\[
	V^*D^\alpha\Gamma_{D^{-\alpha}\phi}V=S_{B}\left(A\right),
	\]
	where $B= \left[  \left(\frac{1+m}{1+m+n} \right)^\alpha\right]_{m,n\ge 0}$. It follows that
	\[
	\left\|D^\alpha\Gamma_{D^{-\alpha}\phi}\right\|_{\LC\left(H^2\left(\HC\right)\right)}\ge \left\|S_{B}\left(A\right)\right\|_{\LC\left(l^2\left(\N_0\right)\right)}.
	\]
	From Bennett's criterion, $S_B$ is not a bounded Schur multiplier, and so the right-hand side in the above inequality will be infinite for some choice of $A$. It follows that, for the same choice of $A$, $D^\alpha\Gamma_{D^{-\alpha}\phi}$ is not bounded on $H^2\left(\HC\right)$.
	
	\section*{Acknowledgments}
	The author expresses his gratitude to Sandra Pott, and Alexandru Aleman, for interesting discussions on the topics above, and also to Erik Wahlén, and the anonymous referee, for their useful comments on the presentation of this manuscript.

\bibliographystyle{amsplain}

\end{document}